\documentclass[12pt]{article}
\usepackage{fullpage} 

\usepackage{ucs}
\usepackage{amssymb}
\usepackage{amsthm}
\usepackage{amsmath}
\usepackage{latexsym}
\usepackage[cp1251]{inputenc}
\usepackage[english]{babel}
\usepackage{graphicx}
\usepackage{wrapfig}
\usepackage{caption}
\usepackage{subcaption}
\usepackage{txfonts}
\usepackage{lmodern}
\usepackage{mathrsfs}
\usepackage{dsfont}
\usepackage{enumerate}
\usepackage{hyperref}
\usepackage{multicol}
\usepackage{csquotes}
\usepackage{stmaryrd}
\usepackage{tikz}
\usepackage{comment}
\usepackage{enumitem}
\usepackage{longtable}
\usepackage{hyperref}
\usepackage{cite}
\usepackage{pifont}
\usepackage{lipsum}
\usetikzlibrary{calc}

\title{Maximum Number of Almost Similar Triangles in the Plane}

\newtheorem{theo}{Theorem}

\newtheorem{lemma}[theo]{Lemma}

\newtheorem{corl}[theo]{Corollary}
\newtheorem{conj}[theo]{Conjecture}
\newtheorem{claim}[theo]{Claim}

\theoremstyle{definition}

\newtheorem{defn}[theo]{Definition}

\numberwithin{theo}{section}

\usepackage[ruled,vlined]{algorithm2e}

\begin{document}

\author{%
  J\'ozsef Balogh \footnote{Department of Mathematics, University of Illinois at Urbana-Champaign, Urbana, Illinois 61801, USA, and Moscow Institute of Physics and Technology, Russian Federation. E-mail: \texttt{jobal@illinois.edu}. Research is partially supported by NSF Grant DMS-1764123, NSF RTG grant DMS 1937241, Arnold O. Beckman Research
Award (UIUC Campus Research Board RB 18132), the Langan Scholar Fund (UIUC), and the Simons Fellowship.}
\and Felix Christian Clemen \footnote {Department of Mathematics, University of Illinois at Urbana-Champaign, Urbana, Illinois 61801, USA, E-mail: \texttt{fclemen2@illinois.edu}.}
 \and Bernard Lidick\'{y} \footnote {Iowa State University, Department of Mathematics, Iowa State University, Ames, IA., E-mail: \newline \texttt{ lidicky@iastate.edu}. Research of this author is partially supported by NSF grant DMS-1855653.}
}
\date{\today}
\maketitle

\abstract{
A triangle $T'$ is $\varepsilon$-similar to another triangle $T$ if their angles pairwise differ by at most $\varepsilon$. 
Given a triangle $T$, $\varepsilon>0$ and $n\in\mathbb{N}$, B\'ar\'any and F\"uredi asked to determine the maximum number of triangles $h(n,T,\varepsilon)$ being $\varepsilon$-similar to $T$ in a planar point set of size $n$. We show that for almost all triangles $T$ there exists $\varepsilon=\varepsilon(T)>0$ such that $h(n,T,\varepsilon)=n^3/24 (1+o(1))$. Exploring connections to hypergraph Tur\'an problems, we use flag algebras and stability techniques for the proof.} 

\medskip
\noindent
\textbf{Keywords:} similar triangles, extremal hypergraphs, flag algebras.
 
\noindent
2020 Mathematics Subject Classification:  52C45, 05D05, 05C65 

\section{Introduction}
Let $T,T'$ be triangles with angles $\alpha\leq \beta \leq \gamma$ and $\alpha'\leq \beta'\leq \gamma'$ respectively. The triangle $T'$ is \emph{$\varepsilon$-similar} to $T$ if $|\alpha-\alpha'|< \varepsilon, |\beta-\beta'|<\varepsilon,$ and $|\gamma-\gamma'|<\varepsilon$. B\'ar\'any and F\"uredi~\cite{MR3953886}, motivated by Conway, Croft, Erd\H{o}s and Guy~\cite{MR527745}, studied the maximum number $h(n,T,\varepsilon)$ of triangles in a planar set of $n$ points that are $\varepsilon$-similar to a triangle $T$. For every $T$ and $\varepsilon=\varepsilon(T)>0$ sufficiently small, B\'ar\'any and F\"uredi~\cite{MR3953886} found the following lower bound construction: Place the $n$ points in three groups with as equal sizes as possible, and each group very close to the vertices of the triangle $T$. Now, iterate this by splitting each of the three groups into three further subgroups of points, see Figure~\ref{fig:iterated construction} 
for an illustration of this construction. Define a sequence $h(n)$ by $h(0)=h(1)=h(2)=0$ and for $n\geq 3$
\begin{align*}
    h(n):=\max\{abc+h(a)+h(b)+h(c): a+b+c=n,\ a,b,c\in\mathbb{N} \}.
\end{align*}

\begin{figure}
\begin{center}
    \begin{tikzpicture}[scale=1.5]
    \draw[fill=black] 
    \foreach \a in {-30,80,210}{
        \foreach \b in {-30,80,210}{
        \foreach \c in {-30,80,210}{
        (\a:1) ++(\b:0.3) ++(\c:0.07) circle(1pt)
        }
        }
    }
    ;
    \end{tikzpicture}
\end{center}
    \caption{Construction sketch on 27 vertices.}
    \label{fig:iterated construction}
\end{figure}
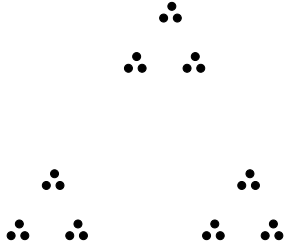

By the previously described construction, this sequence $h(n)$ is a lower bound on $h(n,T,\varepsilon)$. 
For $T$ being an equilateral triangle equality holds.
\begin{theo}[B\'ar\'any, F\"uredi\cite{MR3953886}]
Let $T$ be an equilateral triangle. There exists $\varepsilon_0\geq 1^{\circ}$ such that for all $\varepsilon \in (0,\varepsilon_0)$ and all $n$ we have $h(n,T,\varepsilon)=h(n)$.
In particular, when $n$ is a power of $3$, $h(n,T,\varepsilon)= \frac{1}{24}(n^3-n)$.
\end{theo}
B\'ar\'any and F\"uredi~\cite{MR3953886} also found various examples of triangles $T$ (e.g.~the isosceles right angled triangle) where $h(n,T,\varepsilon)$ is larger than $h(n)$. 

The space of triangle shapes $S\subseteq \mathbb{R}^3$ can be represented
with triples $(\alpha, \beta, \gamma)\in \mathbb{R}^3$ of angles $\alpha,\beta,\gamma>0$ with $\alpha+\beta+\gamma=\pi$. When we make statements about almost every triangle, we mean it in a measure theoretic sense, i.e.~that there exists a set $S'\subseteq S$ with the $2$-dimensional Lebesque measure being $0$ such that the statements holds for all triangles $T\in S\setminus S'$. In \cite{MR3953886} it also was proved that $h(n,T,\varepsilon)$ can only be slightly larger than $h(n)$ for almost every triangle $T$.
\begin{theo}[B\'ar\'any, F\"uredi~\cite{MR3953886}]
\label{pointinplaneasymptotic}
For almost every triangle $T$ there is an $\varepsilon>0$ such that
\begin{align*}
    h(n,T,\varepsilon)\leq 0.25072 \binom{n}{3}(1+o(1)).
\end{align*}
\end{theo}
The previously described construction gives a lower bound of $0.25\binom{n}{3}(1+o(1))$. B\'ar\'any and F\"uredi~\cite{MR3953886} reduced the problem of determining $h(n,t,\varepsilon)$ to a hypergraph Tur\'an problem and used the method of flag algebras, to get an upper bound on the corresponding Tur\'an problem. Flag algebras is a powerful tool invented by Razborov~\cite{flagsRaz}, which has been used to solve problems in various different areas, including graph theory~\cite{FAgraphs,FAgraphs2}, permutations~\cite{PAPerm1,PAPerm2} and discrete geometry~\cite{FAGeometry,FAGeom2}.
An obstacle B\'ar\'any and F\"uredi~\cite{MR3953886} encountered is that the conjectured extremal example is an iterative construction and flag algebras tend to struggle with those. We will overcome this issue by using flag algebras only to prove a weak stability result and then use cleaning techniques to identify the recursive structure. Similar ideas have been used in \cite{MR3425964} and \cite{MR3667664}. This allows us to prove the asymptotic result and for large enough $n$ an exact recursion. 
\begin{theo}
\label{pointsinplanemainasymp}
For almost every triangle $T$ there is an $\varepsilon=\varepsilon(T)>0$ such that
\begin{align}
    h(n,T,\varepsilon)= \frac{1}{4} \binom{n}{3}(1+o(1)).
\end{align}
\end{theo}
\begin{theo}\label{pointinplanemainrekursion}
There exists $n_0$ such that for all $n\geq n_0$ and for almost every triangle $T$ there is an $\varepsilon=\varepsilon(T)>0$ such that
\begin{align}
\label{pointsinplanerecformula}
    h(n,T,\varepsilon)= a\cdot b \cdot c+h(a,T,\varepsilon)+h(b,T,\varepsilon)+h(c,T,\varepsilon),
\end{align}
where $n=a+b+c$ and $a,b,c$ are as equal as possible.
\end{theo}

We will observe that Theorem~\ref{pointinplanemainrekursion} implies the exact result when $n$ is a power of $3$.
\begin{corl}
\label{pointsinplanecorol}
Let $n$ be a power of $3$. Then, for almost every triangle $T$ there is an $\varepsilon=\varepsilon(T)>0$ such that
\begin{align*}
    h(n,T,\varepsilon)= \frac{1}{24}(n^3-n).
\end{align*}
\end{corl}
The paper is organized as follows. In Section~\ref{pointsinplanepreperation} we introduce terminology and notation that we use, we establish a connection from maximizing the number of similar triangles to Tur\'an problems; and we apply flag algebras in our setting to derive a weak stability result. In Section~\ref{pointsinplanemainsec} we apply cleaning techniques to improve the stability result and derive our main results. Finally, in Section~\ref{pointsinplaneconcludingremarks} we discuss further questions. 

\section{Preparation}
\label{pointsinplanepreperation}
\subsection{Terminology and Notation}
\begin{defn}
Let $G$ be a $3$-uniform hypergraph (shortly a $3$-graph), $\mathcal{H}$ be a family of $3$-graphs, $v\in V(G)$ and $A,B\subseteq V(G)$. Then,
\begin{itemize}[leftmargin=*]
\setlength\itemsep{0em}
\item $G$ is \emph{$\mathcal{H}$-free}, if it does not contain a copy of any $H\in \mathcal{H}$.
\item a $3$-graph $G$ on $n$ vertices is \emph{extremal} with respect to $\mathcal{H}$, if $G$ is $\mathcal{H}$-free and $e(G')\leq e(G)$ for every $\mathcal{H}$-free 3-graph $G'$ on $n$ vertices. If it is clear from context, we only say $G$ is extremal.
\item for $a,b\in V(G)$, denote $N(a,b)$ the \emph{neighborhood} of $a$ and $b$, i.e.~the set of vertices $c\in V(G)$ such that $abc\in E(G)$.
\item we write $L(v)$ for the \emph{linkgraph} of $v$, that is the graph $G'$ with $V(G')=V(G)\setminus\{v\}$ and $E(G')$ being the set of all pairs ${a,b}$ with $abv\in E(G)$.
\item we write $L_A(v)$ for the linkgraph of $v$ on $A$, that is the graph $G'$ with $V(G')=A\setminus\{v\}$ and $E(G')$ being the set of all pairs ${a,b}\subseteq A\setminus \{v\}$ with $abv\in E(G)$.
\item we write $L_{A,B}(v)$ for the (bipartite) linkgraph of $v$ on $A\cup B$, that is the graph $G'$ with $V(G')=A\cup B\setminus\{v\}$ and $E(G')$ being the set of all pairs ${a,b}$ with $a\in A, b\in B$ and $abv\in E(G)$.
\item we denote by $|L(v)|,|L_A(v)|$ and $|L_{A,B}(v)|$ the number of edges of the linkgraphs $L(v),L_A(v)$ and $L_{A,B}(v)$ respectively. 
\end{itemize}
\end{defn}
Define a $3$-graph $S(n)$ on $n$ vertices recursively. For $n=1,2$, let $S(n)$ be the $3$-graph on $n$ vertices with no edges. For $n\geq 3$, choose $a\geq b \geq c$ as equal as possible such that $n=a+b+c$. Then, define $S(n)$ to be the $3$-graph constructed by taking vertex disjoint copies of $S(a), S(b)$ and $S(c)$ and adding all edges with all $3$ vertices coming from a different copy. B\'ar\'any and F\"uredi~\cite{MR3953886} observed that $|S(n)|\geq \frac{1}{24}n^3-O(n \log n)$.

Given a set $B\subseteq \mathbb{C}$ and $\delta>0$, we call the set $U_\delta(B):=\{z: |z-b|<\delta \text{ for some } b\in B\}$ the $\delta$-\emph{neighborhood} of $B$. If $B=\{b\}$ for some $b\in \mathbb{C}$, abusing notation, we write $U_\delta(b)$ for it. 

\subsection{Forbidden subgraphs}
\label{pointsinplaneforbidden}

Given a finite point set $P\subseteq \mathbb{R}^2$ in the plane, a triangle $T\in S$ and an $\varepsilon>0$, we denote $G(P,T,\varepsilon)$ the $3$-graph with vertex set $V(G(P,T,\varepsilon))=P$ and triples $abc$ being an edge in $G(P,T,\varepsilon)$ iff $abc$ forms a triangle $\varepsilon$-similar to $T$.
A $3$-graph $H$ is called \emph{forbidden} if $|V(H)|\leq 12$ and for almost every triangle shape $T\in S$ there exists an $\varepsilon=\varepsilon(T)>0$ such that for every point set $P\subseteq \mathbb{R}^2$, $G(P,T,\varepsilon)$ is $H$-free. Denote $\mathcal{F}$ the family of all forbidden $3$-graphs and $\mathcal{T}_\mathcal{F}\subseteq S$ the set of all triangles $T$ such that there exists $\varepsilon=\varepsilon(T)>0$ such that for every point set $P\subseteq \mathbb{R}^2$, $G(P,T,\varepsilon)$ is $\mathcal{F}$-free. Given $T\in \mathcal{T}_\mathcal{F}$, we denote some $\varepsilon(T)>0$ to be a positive real number such that for every point set $P\subseteq \mathbb{R}^2$, $G(P,T,\varepsilon(T))$ is $\mathcal{F}$-free.

In our definition of forbidden $3$-graphs we restrict the size to be at most $12$. The reason we choose the number $12$ is that the largest forbidden subgraph we need for our proof has size $12$ and we try to keep the family $\mathcal{F}$ to be small. 

We will prove Theorem~\ref{pointsinplanemainasymp}, Theorem~\ref{pointinplanemainrekursion} and Corollary~\ref{pointsinplanecorol} for all triangles $T\in \mathcal{T}_\mathcal{F}$.  Note that by the definition of $\mathcal{F}$, almost all triangles are in $\mathcal{T}_\mathcal{F}$. B\'ar\'any and F\"uredi~\cite{MR3953886} determined the following hypergraphs to be members of $\mathcal{F}$.

\begin{lemma}[B\'ar\'any and F\"uredi~\cite{MR3953886}, see Lemma 11.2]
\label{pointsinplaneforbiden1}
 The following hypergraphs are members of $\mathcal{F}$.\\
 \begin{minipage}[t]{.4\textwidth}
\begin{itemize}
    \item $K_4^-=\{123,124,134\}$
    \item $C_5^-=\{123,124,135,245\}$
    \item  $C_5^+=\{126,236,346,456,516\}$
    \item $L_2=\{123,124,125,136,456\}$
    \item $L_3=\{123,124,135,256,346\}$
\end{itemize}
\end{minipage}
\begin{minipage}[t]{.5\textwidth}
    \begin{itemize}[leftmargin=*]
    \item $L_4=\{123,124,156,256,345\}$
    \item $L_5=\{123,124,135,146,356\}$
    \item $L_6=\{123,124,145,346,356\}$
    \item  $P_7^-=\{123,145,167,246,257,347\}.$
\end{itemize}
\end{minipage}
\end{lemma}
\noindent
For the non-computer assisted part our proof, we will need to extend this list. For the computer assisted part, we excluded additional graphs on $7$ and $8$ vertices.
\begin{lemma}
\label{pointsinplaneforbiden2}
 The following hypergraphs are members of $\mathcal{F}$.
\begin{itemize}
    \item  $L_7=\{123,124,125,136,137,458,678\}$
    \item $L_8=\{123,124,125,136,137,468,579,289\}$
    \item $L_9=\{123,124,125,136,237,469,578,189\}$
    \item $L_{10}=\{123, 124, 125, 126, 137, 138, 239, 58a, 47b, 69c, abc\}.$
\end{itemize}
\end{lemma}
Note that this is not the complete list.
To verify that those hypergraphs are forbidden, we will we use the same method as B\'ar\'any and F\"uredi~\cite{MR3953886} used to show that the hypergraphs from Lemma~\ref{pointsinplaneforbiden1} are forbidden. For sake of completeness, we repeat their argument here. 
\begin{proof}
We call a $3$-graph $H$ on $r$ vertices $\emph{dense}$ if there exists a vertex ordering $v_1,v_2,\ldots, v_{r}$ such that for every $3\leq i \leq r-1 $ there exists exactly one edge $e_i\in E(H[\{v_1,\ldots,v_i\}])$ containing $v_i$, and there exists exactly two edges $e_{r}, e_{r}'$ containing $v_{r}$. Note that $L_7,L_8,L_9$ and $L_{10}$ are dense.

For convenience, we will work with a different representation of triangles shapes. A triangle shape $T\in S$ is characterized by a complex number $z \in \mathbb{C} \setminus \mathbb{R}$ such that the triangle with vertices $0,1,z$ is similar to $T$. Note that there are at most twelve complex numbers $w$ such that the triangle $\{0, 1, w\}$ is similar to $T$.

Let $H$ be a dense hypergraph on $r$ vertices with vertex ordering $v_1,\ldots, v_r$ and let $P=\{p_1,\ldots, p_r\}$ $\subseteq \mathbb{R}^2$ be a point set such that $G(P,T,\varepsilon)$ contains $H$ (with $p_i$ corresponding to $v_i$), where $\varepsilon$ is small enough such that the following argument holds. 
Let $\delta>0$ be sufficiently small.
Without loss of generality, we can assume that $p_1=(0,0)$ and $p_2=(1,0)$. Now, since $H$ is dense, $v_1v_2v_3\in E(H)$ and therefore $p_1p_2p_3$ forms a triangle $\varepsilon$-similar to $T$. Therefore, there exists at most $12$ points (which are functions in $z$) such that $p_3$ is in a $\delta$-neighborhood of one of them. Since, $v_4$ is contained in some edge with vertices from $\{v_1,v_2,v_3,v_4\}$, there are at most $12 \cdot 12=144$ points (which are functions in $z$) such that $q_4$ is in a $\delta$-neighborhood of one of them. Continuing this argument, we find functions $f_{i,j}(z)$ in $z$ where $3\leq i \leq r-1$ and    $j \leq 12^{r-3}$ such that
\begin{align*}(p_3, p_4,\ldots, p_{r})\in U_\delta(f_{3,j}(z)) \times  U_\delta(f_{4,j}(z)) \times \ldots \times U_\delta(f_{r-1,j}(z)) 
\end{align*}
for some $j \leq 12^{r-3}$.
Since $H$ is dense, $v_{r}$ is contained in exactly two edges $e_{r}$ and $e_{r}'$. For each $j \leq 12^{r-3}$, because $v_k\in e_{r}$, there exists at most $12$ points $f_{r,j,\ell}(z)$ where $\ell \leq 12$ such that 
\begin{align*}
    p_k \in U_\delta \left(f_{r,j,\ell}(z)\right).
\end{align*}
 Similarly, because $v_k\in e_{r}'$, there exists at most $12$ points $g_{r,j,\ell}(z)$ where $\ell' \leq 12$ such that 
\begin{align*}
    p_k \in U_\delta \left(g_{r,j,\ell'}(z)\right).
\end{align*}
Thus,
\begin{align}
\label{pointsinplanebigset}
p_k \in \bigcup_{\ell,\ell'\leq 12} U_\delta \left(f_{r,j,\ell}(z)\right) \cap U_\delta \left(g_{r,j,\ell'}(z)\right).
\end{align}
Note that if there exists a $z$ such that for each $1\leq j\leq 12^{r-3}$ none of the equations 
\begin{align}
\label{pointsinplaneequations}
    f_{r,j,\ell}(z)=g_{r,j,\ell'}(z), \quad \quad 1\leq \ell,\ell' \leq 12
\end{align}
hold, then we can choose $\varepsilon>0$ such that
\begin{align}
\label{pointsinplanedelta}
\delta< \frac{1}{3} \max_{\ell,\ell'}|f_{r,j,\ell}(z)-g_{r,j,\ell'}(z)|,
\end{align}
and therefore the set in \eqref{pointsinplanebigset} is empty, contradicting that $G(P,T,\varepsilon)$ contains a copy of $H$.
Note that, because of \eqref{pointsinplanedelta}, $\varepsilon$ depends on $z$ and therefore on $T$. 
If we could find one $z\in \mathbb{C}$ not satisfying any of the  equations in \eqref{pointsinplaneequations}, then each of the equations is non-trivial (the solution space is not $\mathbb{C}$). Thus, for each equation the solution set has Lebesque measure 0. Since there are only at most $12^{r-2}$ equations, the union of the solution sets still has measure $0$. Thus, we can conclude that for almost all triangles $T$ there exists $\varepsilon$ such that $G(P,T,\varepsilon)$ is $H$-free for every point set $P$. It remains to show that for $H\in \{L_7,L_8,L_9,L_{10}\}$ there exists $z\in \mathbb{C}$ not satisfying any of the equations in \eqref{pointsinplaneequations}. We will show this for a $z$ corresponding to the equilateral triangle ($z= \frac{1}{2}+i \cdot \frac{\sqrt{3}}{2}$). For $T$ being the equilateral triangle, there are at most $2^{r-2}$ equations to check. Because of the large amount of cases, we will use a computer to verify it. 

Our computer program is a simple brute force recursive approach. It starts by embedding $p_1=(0,0)$ and $p_2=(1,0)$. For each subsequent $3 \leq i \leq r$ it tries both options for embedding $p_i$ dictated by $e_i$. Finally, it checks if the points forming $e'_{r}$ form an equilateral triangle.
If in none of the $2^{r-2}$ generated point configurations the points of $e'_{r}$ form an equilateral triangle, then $H$ is a member of $\mathcal{F}$.
An implementation of this algorithm in python is available at \url{http://lidicky.name/pub/triangle}.
This completes the proof of Lemma~\ref{pointsinplaneforbiden2}.
\end{proof}

Instead of Theorem~\ref{pointsinplanemainasymp} we will actually prove the following stronger result.
\begin{theo}
\label{pointinplaneturanmainrekursion}
We have
\begin{align*}\textup{ex}(n,\mathcal{F})= 0.25\binom{n}{3}(1+o(1)).
\end{align*}
\end{theo}

First, we observe that Theorem~\ref{pointinplaneturanmainrekursion} implies Theorem~\ref{pointsinplanemainasymp}. Let $P\subseteq \mathbb{R}^2$ be a point set of size $n$ and let $T\in \mathcal{T}_{\mathcal{F}}$. Then, $G(P,T,\varepsilon(T))$ is $\mathcal{F}$-free. Now, the number of $\varepsilon$-similar triangles $T$ equals the number of edges in $G(P,T,\varepsilon(T))$. Since $G(P,T,\varepsilon(T))$ is $\mathcal{F}$-free, we have 
\begin{align*}
h(n,T,\varepsilon)\leq \textup{ex}(n,\mathcal{F}).
\end{align*}
Therefore, Theorem~\ref{pointinplaneturanmainrekursion} implies Theorem~\ref{pointsinplanemainasymp}.

\subsection{A structural result via Flag Algebras}
It is a standard application of flag algebras to determine an upper bound for $\textup{ex}(n,\mathcal{G})$ given a family $\mathcal{G}$ of 3-uniform hypergraphs. Running the method of flag algebras on $7$ vertices, B\'ar\'any and F\"uredi~\cite{MR3953886} obtained  
\begin{align}
\label{pointsinplaneturanupper}
    \textup{ex}(n,\mathcal{F})\leq \textup{ex}(n,\{K_4^-,C_5^-,C_5^+,L_2,L_3,L_4,L_5,L_6,P_7^-\})\leq 0.25072 \binom{n}{3}(1+o(1)).
\end{align}
It is conjectured in~\cite{RavryTuran} that $\textup{ex}(n,\{K_4^-,C_5\}) = 0.25 \binom{n}{3}(1+o(1))$.
We note that when running flag algebras on $8$ vertices and forbidding more $3$-graphs in $\mathcal{F}$, then we can obtain the following improved bound.
\begin{align}
    \textup{ex}(n,\mathcal{F}))\leq 0.2502 \binom{n}{3}(1+o(1)).
    \label{ourbound}
\end{align}
Note that Conjecture~\ref{conj:frv} is a significant strengthening of \eqref{pointsinplaneturanupper} and \eqref{ourbound}.
We use flag algebras to prove a stability result. For an excellent explanation of flag algebras in the setting of $3$-graphs see~\cite{RavryTuran}. Here, we will focus on the formulation of the problem rather than providing a formal explanation of the general method.
As a consequence, we obtain the following lemma, which gives the first rough structure of extremal constructions.
This approach was developed in~\cite{MR3425964} and~\cite{MR3667664}.

\begin{lemma}
\label{flaginductive}
Let $n\in \mathbb{N}$ be sufficiently large and let $G$ be an $\mathcal{F}$-free $3$-graph on $n$ vertices and $|E(G)|\geq 1/24 n^3(1+o(1))$ edges. Then there exists an edge $x_1x_2x_3\in E(G)$ such that for $n$ large enough
  \begin{enumerate}
      \item[\textup{(i)}] the neighborhoods $N(x_1,x_2),N(x_2,x_3)$, and $N(x_1,x_3)$ are pairwise disjoint. 
      \item[\textup{(ii)}]  $\min\{|N(x_1,x_2)|, |N(x_2,x_3)|, |N(x_1,x_3)|\} \geq 0.26n.$
      \item[\textup{(iii)}]       
      $n - |N(x_1,x_2)|-|N(x_2,x_3)|-|N(x_1,x_3)| \leq 0.012n.$      
  \end{enumerate}
\end{lemma}
\begin{proof}
Denote $T_{i,j,k}$ the family of $3$-graphs that are obtained from a complete $3$-partite $3$-graph with part sizes $i$, $j$ and $k$ by adding $\mathcal{F}$-free $3$-graphs in each of the three parts. Let $X$ be a subgraph of $G$ isomorphic to $T_{2,2,1}$ on vertices $x_1,x_1',x_2,x_2',x_3$  with edges
$
x_1x_2x_3, x_1x_2'x_3, x_1'x_2x_3, x_1'x_2'x_3
$. Further, define
\begin{align*}
A_1&:=N(x_2,x_3) \cap N(x_2',x_3), &  
A_3&:=N(x_1,x_2) \cap N(x_1',x_2)  \cap N(x_1,x_2') \cap N(x_1',x_2'), \\
 A_2&:=N(x_1,x_3) \cap N(x_1',x_3),
 &
J&:= V(G) \setminus \left( A_1 \cup A_2 \cup A_3 \right).
\end{align*}
Let $a_i := |A_i|/n$ for $1 \leq i \leq 3$. Note that $V(G)=A_1 \cup A_2 \cup A_3 \cup J$ is a partition, because the sets $N(x_1,x_2),N(x_1,x_3)$ and $N(x_2,x_3)$ are pairwise disjoint. Indeed, without loss of generality, assume $N(x_1,x_2) \cap N(x_1,x_3)\neq \emptyset$. Let $v\in N(x_1,x_2) \cap N(x_1,x_3)$. Then $v,x_1,x_2,x_3$ spans at least $3$ edges and therefore $G$ contains a copy of $K_4^-$, a contradiction. We choose $X$ such that
\tikzset{vtx/.style={inner sep=1.1pt, outer sep=0pt, circle, fill,draw}}
\tikzset{lab/.style={inner sep=1.5pt, outer sep=0pt, draw}}

\begin{align}
a_1a_2+a_1a_3+a_2a_3 - \frac{1}{4}\left(a_1^2 + a_2^2 + a_3^2 \right) \label{eq:ai}
\end{align}
is maximized.
Flag algebras can be used to give a lower bound on the expected value of \eqref{eq:ai} for $X$ chosen uniformly at random and therefore also a lower bound on \eqref{eq:ai} when $X$ is chosen to maximize \eqref{eq:ai}. 

Let $Z$ be a fixed \emph{labeled} subgraph of $G$ belonging to $T_{i',j',k'}$.
Denote by $T_{i,j,k}(Z)$ the family of subgraphs of $G$ that contain $Z$, belong to $T_{i,j,k}$, where $i' \leq i$, $j' \leq j$, and $k' \leq k$, and the natural three parts of $Z$ are mapped to the same 3 parts in $T_{i,j,k}(Z)$. The normalized number of $T_{i,j,k}(Z)$ is  
\[
t_{i,j,k}(Z) :=\frac{|T_{i,j,k}(Z)|}{ \binom{n-|V(Z)|}{i+j+k-|V(Z)|}}.
\]
The subgraphs of $G$ isomorphic to $T_{i,j,k}$ are denoted by $T_{i,j,k}(\emptyset)$.
The normalized number is
\[
t_{i,j,k} := \frac{|T_{i,j,k}(\emptyset)|}{\binom{n}{i+j+k}}.
\]

Notice that $a_1  = t_{3,2,1}(X) + o(1)$, 
$2a_1a_2  = t_{3,3,1}(X) + o(1)$, and $a_1^2 = t_{4,3,1}(X)+o(1)$.
We start with \eqref{eq:ai} and obtain the following.
 \begingroup
\allowdisplaybreaks
\begin{align*}
    &~  \left(a_1a_2+a_1a_3+a_2a_3 - \frac{1}{4}\left(a_1^2 + a_2^2 + a_3^2 \right) \right) n^2\\
    =&~  \left(2a_1a_2+2a_1a_3+2a_2a_3 - \frac{1}{2}\left(a_1^2 + a_2^2 + a_3^2 \right) \right) \binom{n-5}{2} + o(n^2)\\
    =&~
    \left(
    t_{3,3,1}(X) + t_{3,2,2}(X) + t_{2,3,2}(X) 
    - \frac{1}{2} \left(
    t_{4,2,1}(X) + t_{2,4,1}(X) + t_{2,2,3}(X)
    \right)
    \right) \binom{n-5}{2}\\ &+o(n^2) \\
    \geq&~
    \frac{1}{t_{2,2,1} \binom{n}{5}}
    \Bigg(
    \sum_{Y \in T_{2,2,1}(\emptyset)}
    (
        t_{3,3,1}(Y) + t_{3,2,2}(Y) + t_{2,3,2}(Y) 
    \\ &- \frac{1}{2} \left(
    t_{4,2,1}(Y) + t_{2,4,1}(Y) + t_{2,2,3}(Y)
    \right)
    \Bigg) \binom{n-5}{2}
    +o(n^2) \\
 \geq&~
    \frac{1}{t_{2,2,1} \binom{n}{5}}
    \left(
    9\, t_{3,3,1} + 12\, t_{3,2,2}
    - \frac{1}{2} \left(
    6\, t_{4,2,1} + 3\, t_{2,2,3}
    \right)
    \right)    
\binom{n}{7}+o(n^2)\\
 =&~  
    \frac{1}{7\,t_{2,2,1} }
    \left(
    3\, t_{3,3,1} + 3.5\, t_{3,2,2}
    -  \,  t_{4,2,1}
    \right)    
\binom{n-5}{2}+o(n^2).
\end{align*}
\endgroup

\begin{claim}
\label{pointsinplaneflagclaim}
Using flag algebras, we get that
if $\, t_{1,1,1} \geq 0.25$ then 
\[
    \frac{1}{7\,t_{2,2,1} }
    \left(
    3\, t_{3,3,1} + 3.5\, t_{3,2,2}
    -  \,  t_{4,2,1}
    \right)  
    \geq \frac{1.2814228}{ 7 \cdot 0.37502377}> 0.48813.
    \]
\end{claim}
The calculations for Claim~\ref{pointsinplaneflagclaim} are computer assisted; we use CSDP~\cite{csdp} to calculate numerical solutions of semidefinite programs. The data files and programs for the calculations are available at \url{http://lidicky.name/pub/triangle}. Claim~\ref{pointsinplaneflagclaim} gives a lower bound on \eqref{eq:ai} as follows
\begin{align}
a_1a_2+a_1a_3+a_2a_3 - \frac{1}{4}\left(a_1^2 + a_2^2 + a_3^2 \right) \geq \frac{1.2814228}{14 \cdot 0.37502377} > 0.24406. \label{eq:a}
\end{align}
Notice that if $a_1=a_2=a_3=\frac{1}{3}$, then  \eqref{eq:ai}, which is the left hand side of \eqref{eq:a}, is $0.25$. The conclusions (ii) and (iii) of Lemma~\ref{flaginductive}
can be obtained from \eqref{eq:a}. Indeed, assume $a_1< 0.26$. Then,
\begin{align*}
&a_1a_2+a_1a_3+a_2a_3 - \frac{1}{4}\left(a_1^2 + a_2^2 + a_3^2 \right) \\
&\leq a_1 (
1-a_1) + \left(\frac{1-a_1}{2}\right)^2- \frac{1}{4}\left(a_1^2 + 2\left(\frac{1-a_1}{2}\right)^2 \right) \\
&= -\frac{9}{8}a_1^2+\frac{3}{4}a_1+\frac{1}{8}< -\frac{9}{8}0.26^2+\frac{3}{4}0.26+\frac{1}{8}=0.24325,
\end{align*}
contradicting \eqref{eq:a}. Thus, we have $a_1\geq 0.26$, concluding (ii). Next, assume $a_1+a_2+a_3 \leq 0.988$. Then,
\begin{align*}
&a_1a_2+a_1a_3+a_2a_3 - \frac{1}{4}\left(a_1^2 + a_2^2 + a_3^2 \right) \\
&\leq a_1 (
0.988-a_1) + \left(\frac{0.988-a_1}{2}\right)^2- \frac{1}{4}\left(a_1^2 + 2\left(\frac{0.988-a_1}{2}\right)^2 \right) \\
&= -\frac{9}{8}a_1^2+0.741a_1+0.122018 \leq \frac{61009}{250000} < 0.244037, 
\end{align*}
where in the last step we used that the maximum is obtained at $a_1=247/750$. This contradicts \eqref{eq:a}. Thus, we have $a_1+a_2+a_3 \geq 0.988$, concluding (iii).
\end{proof}

In the proof of Lemma~\ref{flaginductive}, we chose a suitable copy of $T_{2,2,1}$ to find the initial 3-partition. One could do the same approach by starting with base $T_{1,1,1}$ instead. However, the resulting bounds would be weaker and not sufficient for the rest of the proof.
This is caused by obtaining a lower bound on \eqref{eq:ai} by taking a random base. 


\section{Proof of Theorem~\ref{pointinplaneasymptotic}}
\label{pointsinplanemainsec}
In this section, we will strengthen our flag algebra result Lemma~\ref{flaginductive} by applying cleaning techniques. 
\subsection{The top layer}

\begin{lemma}
\label{pointsinplanepartition}
  Let $G$ be an $\mathcal{F}$-free $3$-graph on $n$ vertices and $|E(G)|\geq 1/24 n^3(1+o(1))$, satisfying $|L(w)|\geq \frac{1}{8}n^2(1+o(1))$ for every $w\in V(G)$. Then there exists an edge $x_1x_2x_3\in E(G)$ such that for 
  \begin{gather*}
      A_1:=N(x_2,x_3), \ \ A_2:=N(x_1,x_3), \ \ A_3:=N(x_1,x_2), \ \ J:=V(G)\setminus (A_1\cup A_2 \cup A_3),\\ A_1':=A_1\setminus\{x_1\}, \ \ A_2':=A_2\setminus\{x_2\}, \ \ \text{and} \ \ A_3':=A_3\setminus\{x_3\}
  \end{gather*}
  we have for $n$ sufficiently large
  \begin{itemize}
      \item[(a)] $ 0.26n\leq |A_i|\leq 0.48n$ for $i\in[3]$.
      \item[(b)] $|J|\leq 0.012n$.
     \item[(c)] No triple $abc$ with $a,b\in A_i'$ and $c\in A_{j}'$ for some $i,j\in[3],i\neq j$ forms an edge.
      \item[(d)] For $v\in V(G)\setminus\{x_1,x_2,x_3\},\ w_1,w_2\in A_i'$, $u_1,u_2\in A_j'$ with $i,j\in[3]$ and $i\neq j$ we have $vw_1w_2\not\in E(G)$ or $vu_1u_2\not\in E(G)$.
      \item[(e)] For every $v\in V(G)\setminus\{x_1,x_2,x_3\}$, there exists $i\in[3]$ such that $|L_{A_j,A_k}(v)|\geq 0.001n^2$, where $j,k\in[3], j\neq k, j\neq i, k\neq i$.
  \end{itemize}
\end{lemma}
\begin{proof}
Apply Lemma~\ref{flaginductive} and get an edge $x_1x_2x_3$ with the properties from Lemma~\ref{flaginductive}. The sets $A_1,A_2,A_3$ are pairwise disjoint. 
\begin{claim}\label{pointsinplaneclaim32}
Properties (a)--(c) holds.
\end{claim}
\begin{proof}
Note that $(a)$ and $(b)$ hold by Lemma~\ref{flaginductive}. To prove $(c)$, assume that there exists $abc\in E(G)$ with $a,b\in A_i'$ and $c\in A_{j}'$ for some $i,j\in[3],i\neq j$. Let $k\in [3], k\neq i, k\neq j$. See Figure~\ref{fig:pointsinplaneclaim32} for an illustration. Now, 
\begin{align*}
x_ix_jx_k,abc,x_jx_ka,x_jx_kb,cx_ix_k\in E(G).
    \end{align*}    

\tikzset{vtx/.style={inner sep=1.7pt, outer sep=0pt, circle, fill,draw}}
\begin{figure}
\centering
\hspace{0.5cm}
\begin{minipage}[b]{0.4\textwidth}
    \begin{tikzpicture}
    \draw
    (90:1) node[vtx,label=above:$x_j$](x1){}
    (210:1) node[vtx,label=left:$x_k$](x2){}
    (-30:1) node[vtx,label=below:$x_i$](x3){}
    (x3) ++(0.5,0) node[vtx,label=below:$a$](a){}
    (x3) ++(1,0) node[vtx,label=below:$b$](b){}
    (x1) ++(0.5,0) node[vtx,label=above:$c$](c)(c){}
    ;
    \draw
    (a)--(b)--(c)--(a)
    (x1)--(x3)--(x2)--(x1)
    (x1) ellipse(1cm and 0.7cm)
    (x2) ellipse(1cm and 0.7cm)
    (a) ellipse(1cm and 0.7cm)
    ;
    \end{tikzpicture}
    \caption{Situation in Claim~\ref{pointsinplaneclaim32}.}
    \label{fig:pointsinplaneclaim32}
 \end{minipage}
\hspace{0.5cm}
\begin{minipage}[b]{0.4\textwidth}   
    \begin{tikzpicture}
    \draw
    (90:1) node[vtx,label=above:$x_j$](x1){}
    (210:1) node[vtx,label=left:$x_k$](x2){}
    (-30:1) node[vtx,label=below:$x_i$](x3){}
    (x3) ++(0.5,0) node[vtx,label=below:$w_1$](w_1){}
    (x3) ++(1,0) node[vtx,label=below:$w_2$](w_2){}
    (x1) ++(0.5,0) node[vtx,label=above:$u_1$](u_1){}
    (x1) ++(1,0) node[vtx,label=above:$u_2$](u_2){}
    (x1) ++(2,-0.5) node[vtx,label=above:$v$](v){}
    ;
    \draw
    (u_1)--(u_2)--(v)--(u_1)
    (w_1)--(w_2)--(v)--(w_1)
    (x1)--(x3)--(x2)--(x1)
    (x2) ellipse(1cm and 0.7cm)
    (w_1) ellipse(1cm and 0.7cm)
    (u_1) ellipse(1cm and 0.7cm)
    ;
    \end{tikzpicture}
    \caption{Situation in Claim~\ref{pointsinplaneclaimd}.}
    \label{fig:pointsinplaneclaim33}

\end{minipage}
\end{figure}    
    
Therefore $G$ contains a copy of $L_2$ on  $\{x_1,x_2,x_3,a,b,c\}$, a contradiction.  
\end{proof}

\begin{claim}
\label{pointsinplaneclaimd}
Property (d) holds.
\end{claim}
\begin{proof}
Towards contradiction, assume that there exists $v\in V(G)\setminus\{x_1,x_2,x_3\}, w_1,w_2\in A_i'$, $u_1,u_2\in A_j'$ for $i,j\in[3]$ with $i\neq j$ such that $vw_1w_2\in E(G)$ and $vu_1u_2\in E(G)$. Let $k\in [3], k\neq i, k\neq j$. See Figure~\ref{fig:pointsinplaneclaim33} for an illustration. Now, $\{x_1,x_2,x_3,v,u_1,u_2,w_1,w_2\}$ spans a copy of $L_7$ because 
\begin{align*}x_ix_jx_k,vw_1w_2,vu_1u_2, x_jx_kw_1, x_jx_kw_2, x_ix_ku_1, x_ix_ku_2\in E(G).
\end{align*}
However, $L_7\in \mathcal{F}$ by Lemma~\ref{pointsinplaneforbiden2}, contradicting that $G$ is $\mathcal{F}$-free.
\end{proof}

\begin{claim}
Property (e) holds.
\end{claim}
\begin{proof}
Let $v\in V(G)\setminus\{x_1,x_2,x_3\}$. Towards contradiction, assume
\begin{align*}
|L_{A_{1},A_{2}}(v)|< 0.001n^2 \quad \text{and} \quad |L_{A_{1},A_{3}}(v)|< 0.001n^2 \quad \text{and} \quad |L_{A_{2},A_{3}}(v)|< 0.001n^2.
\end{align*} 
By property $(d)$, there exists $i\in[3]$ such that $|L_{A_j'}(v)|=|L_{A_k'}(v)|=0$ for $j,k\in[3]\setminus\{i\}$ with $j\neq k$.
Note, that $|L_{A_i}(v)|\leq |A_i|^2/4$, since $L_{A_i}(v)$ is triangle-free, because otherwise there was a copy of $K_4^-$ in $G$. We have
\begin{align*}
    |L(v)|&\leq |J|\cdot n + |L_{A_{1},A_{2}}(v)| +|L_{A_{2},A_{3}}(v)| +|L_{A_{1},A_{3}}(v)| \\
    &+ |L_{A_1}(v)|+|L_{A_2}(v)|+|L_{A_3}(v)|\\
    &\leq |J| \cdot n +0.003n^2+\frac{|A_i|^2}{4}+2n
    \leq 0.012n^2+0.003n^2+ 0.06n^2+2n\\
    &< 0.08n^2(1+o(1)),
\end{align*}
contradicting the assumption $|L(v)| \geq \frac{1}{8}n^2(1+o(1))$. Note that we used $|A_i|\leq 0.48n$ and $|J|\leq 0.012n$ from properties $(a)$ and $(b)$. 
\end{proof}

This completes the proof of Lemma~\ref{pointsinplanepartition}.
\end{proof}

\begin{lemma}
\label{Fextrpartition}
  Let $n\in \mathbb{N}$ be sufficiently large and let $G$ be an $\mathcal{F}$-free $3$-graph on $n$ vertices and $|E(G)|\geq 1/24 n^3(1+o(1))$, satisfying $|L(w)|\geq\frac{1}{8}n^2(1+o(1))$ for every $w\in V(G)$. Then there exists a vertex partition $V(G)=X_1 \cup X_2 \cup X_3$ with $|X_i|\geq 0.26n$ for $i\in[3]$ such that no triple $abc$ with $a,b\in X_i$ and $c\in X_{j}$ for some $i,j\in[3]$ with $i\neq j$ forms an edge.
\end{lemma}
\begin{proof}
Let $x_1x_2x_3\in E(G)$ be an edge with the properties from Lemma~\ref{pointsinplanepartition}. By property (e) we can partition $J=J_1 \cup J_2 \cup J_3$ such that for every $v\in J_i$
we have $|L_{A_j,A_k}(v)|\geq 0.001n^2$, where $j,k\in[3], j\neq k, j\neq i, k\neq i$.
Set $X_i:=A_i\cup J_i$. Note that by properties (c) and (e) for every $v\in X_i\setminus\{x_i\}$ we have $|L_{A_j,A_k}(v)|\geq 0.001n^2$, where $j,k\in[3], j\neq k, j\neq i, k\neq i$. 
Further, by property (a) and definition of $X_i$ we have $|X_i|\geq 0.26n$ for $n$ large enough.

Towards contradiction, assume that there exists $a,b\in X_1$ and $c\in X_2$ with $abc\in E(G)$.
For each $a,b,c$ we find their neighbors in $A_1\cup A_2\cup A_3$ that put them to $J_1$ and $J_2$. These neighbors are in $A_1\cup A_2\cup A_3$ because they were adjacent to some of $x_1,x_2,x_3$. This will eventually form one of the forbidden subgraphs.
We will distinguish cases depending on how $a,b,c$ coincide with $x_1,x_2,x_3$.

\tikzset{vtx/.style={inner sep=1.7pt, outer sep=0pt, circle, fill,draw}}
\begin{figure}
\centering
\hspace{0.5cm}
\begin{minipage}[b]{0.4\textwidth}
    \begin{tikzpicture}
    \draw
    (90:1) node[vtx,label=above:$x_2$](x1){}
    (210:1) node[vtx,label=below:$x_3$](x2){}
    (-30:1) node[vtx,label=below:$x_1$](x3){}
    (x3) ++(1.0,0) node[vtx,label=below:$a$](a){}
    (x3) ++(1.5,0) node[vtx,label=below:$b$](b){}
    (x3) ++(0.5,0) node[vtx,label=below:$c_1$](c_1){}
    (x1) ++(0.5,0) node[vtx,label=above:$c$](c)(c){}
    (x2) ++(-0.5,0) node[vtx,label=below:$a_3$](a_3){}
    (x2) ++(-1,0) node[vtx,label=below:$b_3$](b_3){}
    (x2) ++(-1.4,0) node[vtx,label=below:$c_3$](c_3){}
    (x1) ++(-0.42,0) node[vtx,label=above:$a_2$](a_2){}
    (x1) ++(-0.84,0) node[vtx,label=above:$b_2$](b_2){}    
    ;
    \draw
    (a)--(b)--(c)--(a)
    (x1)--(x2)--(x3)--(x1)
    (x1) ellipse(1.3cm and 0.8cm)
    (a_3) ellipse(1.3cm and 0.8cm)
    (a) ellipse(1.3cm and 0.8cm)
    ;
        
    \end{tikzpicture}
    \caption{Case 1.}
    \label{fig:pointsinplanecase1}
\end{minipage}
\hspace{0.5cm}
\begin{minipage}[b]{0.4\textwidth}
     \begin{tikzpicture}
    \draw
    (90:1) node[vtx,label=above:$x_2$](x1){}
    (210:1) node[vtx,label=below:$x_3$](x2){}
    (-30:1) node[vtx,label=below:$x_1$](x3){}
    (x3) ++(1,0) node[vtx,label=below:$c_1$](c_1){}
    (x3) ++(0.5,0) node[vtx,label=below:$b$](b){}
    (x1) ++(0.5,0) node[vtx,label=above:$c$](c)(c){}
    (x2) ++(-0.5,0) node[vtx,label=below:$b_3$](b_3){}
    (x2) ++(-1,0) node[vtx,label=below:$c_3$](c_3){}
    (x1) ++(-0.5,0) node[vtx,label=above:$b_2$](b_2){}    
    ;
    \draw
    (x3)--(b)--(c)--(x3)
    (x1)--(x2)--(x3)--(x1)
    (x1) ellipse(1.3cm and 0.8cm)
    (a_3) ellipse(1.3cm and 0.8cm)
    (b) ellipse(1.3cm and 0.8cm)
    ; 
    \end{tikzpicture}
    \caption{Case 2.}
    \label{fig:pointsinplanecase2}
\end{minipage}

\begin{minipage}[b]{0.4\textwidth}
     \begin{tikzpicture}
    \draw
    (90:1) node[vtx,label=above:$x_2$](x1){}
    (210:1) node[vtx,label=below:$x_3$](x2){}
    (-30:1) node[vtx,label=below:$x_1$](x3){}
    (x3) ++(0.5,0) node[vtx,label=below:$a$](a){}
    (x3) ++(1,0) node[vtx,label=below:$b$](b){}
    (x2) ++(-0.5,0) node[vtx,label=below:$a_3$](a_3){}
    (x2) ++(-1,0) node[vtx,label=below:$b_3$](b_3){}
    (x1) ++(-0.5,0) node[vtx,label=above:$a_2$](a_2){}
    (x1) ++(-1,0) node[vtx,label=above:$b_2$](b_2){}    
    ;
    \draw
    (a)--(b)--(x1)--(a)
    (x1)--(x2)--(x3)--(x1)
    (a_2) ellipse(1.2cm and 0.7cm)
    (a_3) ellipse(1.2cm and 0.7cm)
    (a) ellipse(1.2cm and 0.7cm)
    ;
      
    \end{tikzpicture}
    \caption{Case 4.}
    \label{fig:pointsinplanecase4}
\end{minipage}
\end{figure}   

\noindent
\textbf{Case 1:} $a,b\neq x_1$ and $c\neq x_2$.

\noindent
Since
\begin{align*}
    |L_{A_2,A_3}(a)|\geq 0.001n^2, \quad |L_{A_2,A_3}(b)|\geq 0.001n^2 \quad \text{and} \quad |L_{A_a,A_3}(c)|\geq 0.001n^2,
\end{align*}
there exists distinct vertices $a_3,b_3,c_3\in A_3, a_2,b_2\in A_2\setminus\{c\}$ and $c_1\in A_1\setminus\{a,b\}$ such that $aa_2a_3,bb_2b_3,cc_1c_3\in E(G)$. See Figure~\ref{fig:pointsinplanecase1} for an illustration. We have
\begin{align*}
    x_1x_2x_3, abc, aa_2a_3,bb_2b_3,cc_1c_3, c_1x_2x_3,a_2x_1x_3,b_2x_1x_3, c_3x_1x_2,b_3x_1x_2,a_3x_1x_2\in E(G),
\end{align*}
and therefore the vertices $\{x_1,x_2,x_3,a,b,c,c_1,a_2,b_2,a_3,b_3,c_3\}$ span a copy of $L_{10}$, a contradiction. 

\noindent
\textbf{Case 2:} $a=x_1$ and $c\neq x_2$.

\noindent
By property $(d)$, there exists distinct vertices $b_3,c_3\in A_3, b_2\in A_2\setminus\{c\}$ and $c_1\in A_1\setminus\{a,b\}$ such that $bb_2b_3,cc_1c_3\in E(G)$. See Figure~\ref{fig:pointsinplanecase2} for an illustration. We have
\begin{align*}
    x_1x_2x_3,x_1bc,bb_2b_3,cc_1c_3, c_1x_2x_3,b_2x_1x_3, c_3x_1x_2,b_3x_1x_2\in E(G),
\end{align*}
and therefore the vertices $\{x_1,x_2,x_3,b,c,c_1,b_2,b_3,c_3\}$ span a copy of $L_{9}$, a contradiction. 

\noindent
\textbf{Case 3:} $b=x_1$ and $c\neq x_2$.

\noindent
This case is similar to Case 2. 

\noindent
\textbf{Case 4:} $a,b\neq x_1$ and $c=x_2$.

\noindent
By property $(d)$, there exists distinct vertices $a_3,b_3\in A_3, a_2,b_2\in A_2\setminus \{c\}$ such that $aa_2a_3,bb_2b_3\in E(G)$. See Figure~\ref{fig:pointsinplanecase4} for an illustration. We have
\begin{align*}
    x_1x_2x_3, abx_2,aa_2a_3,bb_2b_3,a_2x_1x_3,b_2x_1x_3,b_3x_1x_2,a_3x_1x_2\in E(G),
\end{align*}
and therefore the vertices $\{x_1,x_2,x_3,a,b,a_2,b_2,a_3,b_3\}$ span a copy of $L_{8}$, a contradiction. 

\noindent
\textbf{Case 5:} $a=x_1$ and $c=x_2$.

\noindent
This means that $b\in N(x_1,x_2)=A_3$, contradicting $b\in X_1$. 

\noindent
\textbf{Case 6:} $b=x_1$ and $c=x_2$.

\noindent
This case is similar to case 5.\\

We conclude that for $a,b\in X_1,c\in X_3$, we have $abc\not\in E(G)$. Similarly, for $a,b\in X_i,c\in X_j$ with $i\neq j$, we have $abc\not\in E(G)$.

\end{proof}

\subsection{The asymptotic result}
In this subsection we will prove Theorem~\ref{pointinplaneturanmainrekursion}. We first observe that an extremal $\mathcal{F}$-free $3$-graph satisfies a minimum degree condition.
\begin{lemma}
\label{pointsinplaneaddvertex}
Let $G$ be an $\mathcal{F}$-free $3$-graph and $v\in V(G)$. Denote $G_{u,v}$ the $3$-graph constructed from $G$ by adding a copy $w$ of $v$ and deleting $u$, i.e. 
\begin{align*}
    V(G_{u,v})=V(G)\cup \{w\} \setminus\{u\}, \quad E(G_{u,v})=E(G[V(G)\setminus \{u\}]) \cup \{wab \ | \ abv\in E(G) \}.
\end{align*}
Then, $G_{u,v}$ is also $\mathcal{F}$-free.
\end{lemma}
\begin{proof}
Towards contradiction assume that $G_{u,v}$ does contain a copy of some $F\in\mathcal{F}$. Since $G$ is $\mathcal{F}$-free, this copy $F'$ of $F$ contains the vertices $v$ and $w$. $F'-w$ is a subgraph of $G$ and thus $\mathcal{F}$-free, in particular $F' - w \notin \mathcal{F}$. Thus, there exists a set of triangles shape $\mathcal{T}$ of positive measure such that for $T\in \mathcal{T}$ and $\varepsilon>0$ there exists a point set $P=P(T,\varepsilon)\subseteq \mathbb{R}^2$ with $F' - w$ being isomorphic to $G(P,T,\varepsilon)$. Construct a new point set $P'$ from $P(T,\varepsilon)$ by adding a new point $p_w$ close enough to the point corresponding to $v$. This guarantees that $v$ and $p_w$ have the same linkgraph in $G(P',T,\varepsilon)$ and that there is no edge in $G(P',T,\varepsilon)$ containing both $p_w$ and $v$. Now, $G(P',T,\varepsilon)$ contains a copy of $F$, contradicting that $F\in\mathcal{F}$.

\end{proof}
\begin{lemma}
\label{pointsplanemindegree}
Let $G$ be an extremal $\mathcal{F}$-free $3$-graph on $n$ vertices. Then for every $w\in V(G)$, we have $|L(w)|\geq \frac{1}{8} n^2(1+o(1))$. 
\end{lemma}
\begin{proof}
Assume that there exists $u\in V(G)$ with $|L(u)|<\frac{1}{8} n^2-n^{3/2}$ for $n$ sufficiently large. Let $v\in V(G)$ be a vertex maximizing $|L(v)|$. The $3$-graph $G_{u,v}$ is $\mathcal{F}$-free by Lemma~\ref{pointsinplaneaddvertex} and has more edges than $G$:
\begin{align*}
    &|E(G_{u,v})|-|E(G)|\geq -|L(u)|+|L(v)|-d(v,u)\geq -\frac{1}{8} n^2+n^{3/2}+\frac{3|E(G)|}{n}-n \\
    \geq& -\frac{1}{8} n^2+n^{3/2}+\frac{3|E(S(n))|}{n}-n \geq -\frac{1}{8} n^2+n^{3/2}+\frac{1}{8}n^3-O(n\log n) >0,
\end{align*}
for $n$ sufficiently large. This contradicts the extremality of $G$. Thus for every $w\in V(G)$, we have $|L(w)|\geq \frac{1}{8} n^2-n^{3/2}= \frac{1}{8} n^2(1+o(1))$. 
\end{proof}

\begin{proof}[Proof of Theorem~\ref{pointinplaneturanmainrekursion}]
For the lower bound, we have
\begin{align*}
    \textup{ex}(n,\mathcal{F})\geq e(S(n))=\frac{1}{24}n^3 (1+o(1)).
\end{align*}
For the upper bound, let $n_0$ be large enough such that the following reasoning holds. For $n\geq n_0$, $\textup{ex}(n,\mathcal{F})\leq 0.251 \binom{n}{3}$ by \eqref{pointsinplaneturanupper}. We will prove by induction on $n$ that $\textup{ex}(n,\mathcal{F})\leq \frac{1}{24}n^3+n \cdot n_0^2$. This trivially holds for $n\leq n_0$, because
\begin{align*}
    \textup{ex}(n,\mathcal{F})\leq \binom{n}{3}\leq \frac{1}{24}n^3+n \cdot n_0^2.
\end{align*}
For $n_0\leq n\leq 4n_0$, we have 
\begin{align*}
     \textup{ex}(n,\mathcal{F})\leq 0.251\binom{n}{3}\leq \frac{1}{24}n^3+0.001 \frac{n^3}{6}\leq \frac{1}{24}n^3+n \cdot n_0^2.
\end{align*}
Now, let $G$ be an extremal $\mathcal{F}$-free $3$-graph on $n\geq 4n_0$ vertices. By Lemma~\ref{pointsplanemindegree} we have $|L(w)|\geq \frac{1}{8}n^2(1+o(1))$ for every $w\in V(G)$. Therefore, the assumptions for Lemma~\ref{Fextrpartition} hold. Take a vertex partition $V(G)=X_1\cup X_2 \cup X_2$ with the properties from Lemma~\ref{Fextrpartition}. Now, for all $i\in[3]$, $|X_i|\geq 0.26 n\geq n_0$ and since $G[X_i]$ is $\mathcal{F}$-free, we have 
\begin{align*}
    e(G[X_i])\leq \frac{1}{24}|X_i|^3+|X_i| \cdot n_0^2
\end{align*}
by the induction assumption. We conclude
\begin{align*}
    e(G)&\leq |X_1||X_2||X_3|+ \sum_{i=1}^3 e(G[X_i])\leq |X_1||X_2||X_3|+ n\cdot n_0^2+ \frac{1}{24}\sum_{i=1}^3 |X_i|^3\\
    &\leq \frac{1}{24}n^3+n \cdot n_0^2,
\end{align*}
where in the last step we used that the function $g(x_1,x_2,x_3):=x_1x_2x_3+1/24 (x_1^3+x_2^3+x_3^3)$ with domain $\{(x_1,x_2,x_3) \in [0.26,0.48]^3 : x_1+x_2+x_3=1\}$ archives its maximum at $x_1=x_2=x_3=1/3$. This can be verified quickly using basic calculus or simply by using a computer, we omit the details.  
\end{proof}
Analyzing the previous proof actually gives a stability result. 
\begin{lemma}
\label{pointsinplanepartition2}
Let $G$ be an $\mathcal{F}$-free $3$-graph on $n$ vertices and $|E(G)|= 1/24 n^3(1+o(1))$, satisfying $|L(w)|\geq \frac{1}{8}n^2(1+o(1))$ for every $w\in V(G)$.
  Then $G$ has a vertex partition $V(G)=X_1\cup X_2 \cup X_2$ such that 
  \begin{itemize}
      \item $|X_i|=\frac{n}{3} (1+o(1))$ for every $i\in[3]$,
      \item there is no edge $e=xyz$ with $x,y\in X_i$ and $z\notin X_i$ for $i\in[3]$.
  \end{itemize}
\end{lemma}
\begin{proof}
Take a vertex partition $V(G)=X_1\cup X_2 \cup X_2$ from Lemma~\ref{Fextrpartition}. Since $G[X_i]$ is $\mathcal{F}$-free, we have by Theorem~\ref{pointinplaneturanmainrekursion} that $e(G[X_i])\leq \frac{1}{24}|X_i|^3 (1+o(1))$. Now, again
\begin{align*}
    \frac{1}{24}n^3(1+o(1))&=e(G)\leq |X_1||X_2||X_3|+ \sum_{i=1}^3 e(G[X_i])\\
    &\leq |X_1||X_2||X_3|+  \frac{1}{24}\sum_{i=1}^3 |X_i|^3(1+o(1).
\end{align*}
Again, since the polynomial $g$ with domain $\{(x_1,x_2,x_3) \in [0.26,0.48]^3 : x_1+x_2+x_3=1\}$ achieves its unique maximum at $x_1=x_2=x_3=1/3$, we get $|X_i|= (1/3+o(1))n$.
\end{proof}

\subsection{The exact result}

\begin{lemma}
\label{pointsinplanepointswap}
Let $T\in \mathcal{T}_\mathcal{F}$ and $P\subseteq \mathbb{R}^2$ be a point set. Denote $G=G(P,T,\varepsilon(T))$. For every $u,v\in V(G)$ there exists a point set $P'$ such that $G_{u,v}=G(P',T,\varepsilon(T))$.
\end{lemma}
\begin{proof}
Let $u,v\in V(G)$. Construct $P'$ from $P$ by removing the point corresponding to $u$ and adding a point close enough to the point corresponding to $v$. This point set satisfies $G_{u,v}=G(P',T,\varepsilon(T))$.
\end{proof}

\begin{lemma}
\label{pointsplanemindegree2}
Let $T\in \mathcal{T}_\mathcal{F}$ be a triangle shape and let $P\subseteq \mathbb{R}^2$ be an $n$-element point set maximizing the number of triangles being $\varepsilon(T)$-similar to $T$. Denote $G=G(P,T,\varepsilon(T))$. Then for every $w\in V(G)$, we have $|L(w)|\geq \frac{1}{8} n^2(1+o(1))$. 
\end{lemma}
\begin{proof}
We have that $G$ is $\mathcal{F}$-free. Assume that there exists $u\in V(G)$ with 
\begin{align*}
    |L(u)|<\frac{1}{8}n^2-n^{3/2}.
\end{align*} 
Let $v\in V(G)$ be a vertex maximizing $|L(v)|$. By Lemma~\ref{pointsinplanepointswap} there exists a point set $P'$ such that $G_{u,v}=G(P',T,\varepsilon(T))$. We have $|E(G_{u,v})|>|E(G)|$ by the same calculation as in the proof of Lemma~\ref{pointsplanemindegree}. This contradicts the maximality of $P$. 
\end{proof}
Now, we will strengthen the previous stability result.
\begin{lemma}
Let $T\in \mathcal{T}_\mathcal{F}$. There exists $n_0$ such that for every $n\geq n_0$ the following holds. Let $P$ be an $n$-element point set maximizing the number of triangles being $\varepsilon(T)$-similar to $T$. Then, the $3$-graph $G=G(P,T,\varepsilon(T))$ has a vertex partition $V(G)=X_1\cup X_2 \cup X_2$ such that 
  \begin{enumerate}[label=(\roman*)]
      \item there is no edge $e=xyz$ with $x,y\in X_i$ and $z\notin X_i$ for $i\in[3]$,
      \item $xyz\in E(G)$ for $x\in X_1, y\in X_2, z\in X_3$, 
      \item $|X_i|-|X_j|\leq 1$ for all $i,j\in[3]$.
  \end{enumerate}
\label{pointsinplanepartition3}
\end{lemma}
\begin{proof}
By Lemma~\ref{pointsplanemindegree2}, for every $w\in V(G)$, $|L(w)|\geq \frac{1}{8} n^2(1+o(1))$. Further, we have 
\begin{align*}
    e(G)\geq e(S(n))=\frac{1}{4}\binom{n}{3}(1+o(1)).
\end{align*}
Therefore, the assumptions from Lemma~\ref{pointsinplanepartition2} hold. Let $V(G)=X_1\cup X_2 \cup X_3$ be a partition having the properties from Lemma~\ref{pointsinplanepartition2}. 
Towards contradiction, assume that there exists $x\in X_1, y\in X_2, z\in X_3$ with $xyz\notin E(G)$. For $i\in [3]$, let $P_i$ be the point set corresponding to the set $X_i$. We have,
\begin{align*}
    e(G[X_i])=e(G(P_i,T,\varepsilon(T)).
\end{align*}
Construct a new point set $P'$ by taking a large enough triangle of shape $T$ and placing each of the point sets $P_i$ close to one of the three vertices of $T$. Using condition (i), this new point set $P'$ satisfies
\begin{align*}
    e(G(P',T,\varepsilon(T)))&=|X_1||X_2||X_3|+ \sum_{i=1}^3 e(G(P_i,T,\varepsilon(T)) \\
    &= |X_1||X_2||X_3|+ \sum_{i=1}^3 e(G[X_i]) > e(G),
\end{align*}
contradicting the maximality of $P$. Therefore, for all $x\in X_1, y\in X_2, z\in X_3$ we have $xyz\in E(G)$. By Theorem~\ref{pointinplaneturanmainrekursion}, we have
\begin{align*}
    \frac{e(G[X_1])}{\binom{|X_1|}{3}}=\frac{1}{4}+o(1) \quad \text{ and } \quad \frac{e(G[X_2])}{\binom{|X_2|}{3}}=\frac{1}{4}+o(1).
\end{align*}
Next, towards contradiction, assume that without loss of generality $|X_1|\geq |X_2|+2$. Let $v_1\in X_1$ be minimizing $|L_{X_1}(v_1)|$ and let $v_2\in X_2$ be maximizing $|L_{X_2}(v_2)|$. 
By the choice of $v_1$ and $v_2$,
\begin{align*}
     |L_{X_1}(v_1)|\leq\frac{3e(G[X_1])}{|X_1|} \quad \text{and} \quad |L_{X_2}(v_2)|\geq\frac{3e(G[X_2])}{|X_2|}.
\end{align*}
 The hypergraph $G_{v_1,v_2}$ is still $\mathcal{F}$-free by Lemma~\ref{pointsinplaneaddvertex} and has more edges than $G$: 
 \begingroup
\allowdisplaybreaks
\begin{align*}
    &|E(G_{v_1,v_2})|-|E(G)|=|X_1||X_3|+|L_{X_2}(v_2)|-|L_{X_1}(v_1)|-|X_2||X_3|-|X_3|\\
    \geq& \frac{3e(G[X_2])}{|X_2|}-\frac{3e(G[X_1])}{|X_1|}+|X_3|(|X_1|-|X_2|-1)\\
    =& \frac{3|e(G[X_2])|X_1|-3e(G[X_1])|X_2|}{|X_1||X_2|}+|X_3|(|X_1|-|X_2|-1)\\
    =& \left( \frac{1}{4}+o(1) \right)\frac{3\binom{|X_2|}{3}|X_1|-3\binom{|X_1|}{3}|X_2|}{|X_1||X_2|}+|X_3|(|X_1|-|X_2|-1)\\
    \geq& \left( \frac{1}{8}+o(1) \right)\frac{|X_2|^3|X_1|-|X_1|^3|X_2|}{|X_1||X_2|}+|X_3|(|X_1|-|X_2|-1)\\
    =& \left( \frac{1}{8}+o(1) \right)(|X_2|^2-|X_1|^2)+|X_3|(|X_1|-|X_2|-1)\\
     =&(|X_1|-|X_2|)\left( |X_3|-(|X_1|+|X_2|)\left(\frac{1}{8}+o(1)\right) \right)-|X_3|\\
          =&(|X_1|-|X_2|)\left( \frac{n}{4}+o(n) \right)-|X_3|
    \geq n \left(\frac{1}{2}+o(1)\right)-\left(\frac{1}{3}+o(1)\right)n> 0.
\end{align*}
\endgroup
\end{proof}

\subsection{Proof of  Theorem~\ref{pointinplanemainrekursion}}
Let $T\in\mathcal{T}_\mathcal{F}$ and $P$ be an $n$-element point set maximizing the number of triangles being $\varepsilon(T)$-similar to $T$. Denote $G=G(P,T,\varepsilon(T))$. By Lemma~\ref{pointsinplanepartition3}, the $3$-graph $G$ has a vertex partition $V(G)=X_1\cup X_2 \cup X_2$ such that $|X_i|-|X_j|\leq 1$ for all $i,j\in[3]$ and there is no edge $e=xyz$ with $xy\in X_i$ and $z\notin X_i$ for $i\in[3]$. Since the sets $X_1,X_2,X_3$ correspond to point sets of the same sizes, we have $e(G[X_i])\leq h(|X_i|,T,\varepsilon(T))$ for $i\in[3]$. Let $a=|X_1|,b=|X_2|$ and $c=|X_3|$. Now,
\begin{align*}
    h(n,T,\varepsilon(T))&=e(G)\leq a \cdot b \cdot c + e(G[X_1])+e(G[X_2])+e(G[X_3])\\
    &\leq a \cdot b \cdot c + h(a,T,\varepsilon(T))+h(b,T,\varepsilon(T))+h(c,T,\varepsilon(T)).
\end{align*}
It remains to show
\begin{align*}
    h(n,T,\varepsilon(T))\geq 
 a \cdot b \cdot c + h(a,T,\varepsilon(T))+h(b,T,\varepsilon(T))+h(c,T,\varepsilon(T)).
\end{align*}
There exists point sets $P_a,P_b,P_c\subseteq \mathbb{R}^2$ of sizes $a,b,c$ respectively, such that 
\begin{align*}
    e(G(P_a,T,\varepsilon(T)))= h(a,T,\varepsilon(T)), \quad \quad e(G(P_b,T,\varepsilon(T)))= h(b,T,\varepsilon(T)) \\ \text{and} \quad \quad e(G(P_c,T,\varepsilon(T)))= h(c,T,\varepsilon(T)).
\end{align*} 
Note that we can assume that $\text{diam}(P_a)=1$, $\text{diam}(P_b)=1$ and $\text{diam}(P_c)=1$, where $\text{diam}(Q)$ of a point set $Q$ is the largest distance between two points in the point set. By arranging the three point sets $P_a,P_b,P_c$ in shape of a large enough triangle $T$, we get a point set $P$ such that 
\begin{align*}
 h(n,T,\varepsilon(T))&\geq e(G(P,T,\varepsilon(T)))=a\cdot b \cdot c+ h(a,T,\varepsilon(T)) + h(b,T,\varepsilon(T)) + h(c,T,\varepsilon(T)),
\end{align*}
completing the proof of Theorem~\ref{pointinplanemainrekursion}.

\subsection{Proof of Corollary~\ref{pointsinplanecorol}}
Let $T$ be a triangle shape such that there exists $\varepsilon(T)$ that \eqref{pointsinplanerecformula} holds. 
By Theorem~\ref{pointinplanemainrekursion}, \eqref{pointsinplanerecformula} holds for almost all triangles. Take a point set $P$ on $3^\ell\geq n_0$ points maximizing the number of triangles being $\varepsilon(T)$-similar to $T$. Denote $H=G(P,T,\varepsilon(T))$. Note that because of scaling invariance we can assume that $\text{diam}(P)$ is arbitrary small. By applying \eqref{pointsinplanerecformula} iteratively, we have 
\begin{align}
\label{pointsinplaneextramli}
    h(3^{\ell+i},T,\varepsilon(T))= 3^i \cdot e(H)+3^{3\ell} \frac{1}{24}\left(3^{3i}-3^i\right)
\end{align} 
for all $i\geq 0$. \\
Now, towards contradiction, assume that there exists a point set $P'\subseteq \mathbb{R}^2$ of $3^k$ points such that the number of triangles similar to $\varepsilon(T)$ is more than $e(S(3^k))$. Let $G=G(P',T,\varepsilon(T)).$ Then, 
\begin{align*}
    e(G)>e(S(3^k))=\frac{1}{24}\left( 3^{3k}-3^k \right).
\end{align*}
Construct a point set $\bar{P}\subseteq \mathbb{R}^2$ of $3^{\ell+k}$ points by taking all points $p_G+p_H$, $p_G\in P',p_H\in P$ where addition is coordinate-wise. Let $\bar{G}:= G(\bar{P},T,\varepsilon(T))$. Since we can assume that $\text{diam}(P')$ is arbitrary small, $\bar{G}$ is the $3$-graph constructed from $G$ by replacing every vertex by a copy of $H$. Now,
\begin{align*}
    e(\bar{G})=   e(G) \cdot 3^{3\ell}+e(H) \cdot 3^k> 3^k \cdot e(H)+3^{3\ell} \frac{1}{24}\left(3^{3k}-3^k\right),
\end{align*}
contradicting \eqref{pointsinplaneextramli}. This completes the proof of Corollary~\ref{pointsinplanecorol}.

\section{Concluding remarks}
\label{pointsinplaneconcludingremarks}
When carefully reading the proof, one can observe that also the following Tur\'an type results hold. Recall that $\mathcal{F}$ is the set of forbidden $3$-graphs defined in Section~\ref{pointsinplaneforbidden}.
\begin{theo}
\label{pointsinplaneturanexact}
\begin{itemize}
The following statements holds.
\item[(a)] There exists $n_0$ such that for all $n\geq n_0$ 
\begin{align*}
    \textup{ex}(n,\mathcal{F})= a\cdot b \cdot c+  \textup{ex}(a,\mathcal{F})+  \textup{ex}(b,\mathcal{F})+  \textup{ex}(c,\mathcal{F}),
\end{align*}
where $n=a+b+c$ and $a,b,c$ are as equal as possible.
\item[(b)] Let $n$ be a power of $3$. Then,
\begin{align*}
    \textup{ex}(n,\mathcal{F})= \frac{1}{24}(n^3-n).
    \end{align*}
\end{itemize}
\end{theo}
It would be interesting to prove the Tur\'an type results, Theorem~\ref{pointsinplanemainasymp} and Theorem~\ref{pointsinplaneturanexact}, for a smaller family of hypergraphs than $\mathcal{F}$. Potentially the following conjecture by Falgas-Ravry and Vaughan could be tackled in a similar way. 
\begin{conj}[Falgas-Ravry and Vaughan~\cite{RavryTuran}]\label{conj:frv}
\begin{align*}
    \textup{ex}(n,\{K_4^-,C_5\})=\frac{1}{4}\binom{n}{3}(1+o(1)).
\end{align*}
\end{conj}
Considering that for our proof it was particularly important that $K_4^-$ and $L_2 = \{123,124,125,136,456\}$ are forbidden, we conjecture that $S(n)$ has asymptotically the most edges among $\{K_4^-,L_2\}$-free $3$-graphs.
\begin{conj}
\begin{align*}
    \textup{ex}(n,\{K_4^-,L_2\})=\frac{1}{4}\binom{n}{3}(1+o(1)).
\end{align*}
\end{conj}
Note that a standard application of flag algebras on 7 vertices shows
\begin{align*}
    \textup{ex}(n,\{K_4^-,L_2\})\leq 0.25074\binom{n}{3}
\end{align*}
for $n$ sufficiently large.  

Theorem~\ref{pointsinplanemainasymp} determines $h(n,T,\varepsilon)$ asymptotically for almost all triangles $T$ and $\varepsilon>0$ sufficiently small. It remains open to determine $h(n,T,\varepsilon)$ for some triangles $T\in S$. B\'ar\'any and F\"uredi~\cite{MR3953886} provided asymptotically better bounds stemming from recursive constructions for some of those triangles. Potentially a similar proof technique to ours could be used to determine $h(n,T,\varepsilon)$ for some of those triangle shapes.

Another interesting question is to change the space, and study point sets in $\mathbb{R}^3$ or even $\mathbb{R}^d$ instead of the plane. Given a triangle $T\in S$, $\varepsilon>0$, $d\geq2$ and $n\in \mathbb{N}$, denote 
$g_d(n,T,\varepsilon)$ the maximum number of triangles in a set of $n$ points from $\mathbb{R}^d$ that are $\varepsilon$-similar to a triangle $T$. Being allowed to use one more dimension might help us to find constructions with more triangles being $\varepsilon$-similar to $T$. 

For an acute triangle $T$ and $d=3$, we can group the $n$ points into four roughly equal sized groups and place each group very close to a vertex of a tetrahedron with each face being similar to $T$.
For a crafty reader, we are including a cutout that leads to a tetrahedron with all sides being the same triangle in Figure~\ref{fig:cutout} on the left.
Each group can again be split up in the same way. Keep doing this iteratively gives us 
\begin{align*}
    g_3(n,T,\varepsilon)\geq\frac{1}{15}n^3(1+o(1))
\end{align*}
for some $\varepsilon>0$.
Note that for almost all acute triangles $T$, 
\[
g_2(n,T,\varepsilon) = h(n,T,\varepsilon)=\frac{1}{24}n^3(1+o(1)) < g_3(n,T,\varepsilon).
\]

\begin{figure}
\begin{center}
\begin{tikzpicture}
    \draw
    (0,0) coordinate(a) -- ++(0:4) coordinate(b)
    (a) -- ++(70:3) coordinate(c) -- (b)
    (0.4,-0.1) node{\ding{35}}
    ;
    \draw[dashed]
    ($(a)!0.5!(b)$) -- ($(a)!0.5!(c)$) -- ($(c)!0.5!(b)$) -- cycle
    ;
\end{tikzpicture}
\hskip 2em
\begin{tikzpicture}
    \draw
    (0,0) coordinate(a) -- ++(0:3) coordinate(b)
    (a) -- ++(30:6) coordinate(c) -- (b)
    (0.4,-0.1) node{\ding{35}}
    ;
    \draw[dashed]
    ($(a)!0.5!(b)$) -- ($(a)!0.5!(c)$) -- ($(c)!0.5!(b)$) -- cycle
    ;
\end{tikzpicture}
\end{center}
    \caption{A cutout of a tetrahedron using an acute triangle on the left. A cutout not giving a tetrahedron coming from an obtuse triangle on the right. Bend along the dashed lines.}
    \label{fig:cutout}
\end{figure}
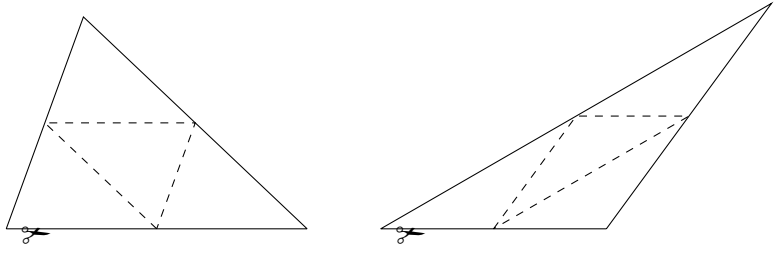

For $T$ being an equilateral triangle and $d\geq4$ we can find a better construction. There is a $d$-simplex with all faces forming equilateral triangles. Grouping the $n$ points into $d+1$ roughly equal sized groups and placing each group very close to the vertex of the $d$-simplex and then iterating this, gives us  
\begin{align*}
    g_d(n,T,\varepsilon)\geq \sum_{i\geq 1}\left(\frac{n}{(d+1)^i}\right)^3 \binom{d+1}{3}(d+1)^{i-1} \ (1+o(1))=\frac{1}{6}\frac{d-1}{d+2}n^3(1+o(1)).
\end{align*}

The following variation of the problem could also be interesting.
We say that two triangles are \emph{$\varepsilon$-isomorphic} if their side lengths are  $a \leq b \leq c$ and 
$a' \leq b' \leq c'$ and $|a-a'|,|b-b'|,|c-c'| < \varepsilon$.
Maximizing the number of $\varepsilon$-isomorphic triangles has the following upper bound. 
Denote the side lengths of a triangle $T$ by $a$, $b$, and $c$. Now color edges of $K_n$ with colors $a$, $b$, and $c$ such that the number of rainbow triangles is maximized. Note that rainbow triangles would correspond to triangles isomorphic  to $T$, if there exists an embedding of $K_n$ in some $R^d$ such that the distances correspond to the colors.
The problem of maximizing the number of rainbow triangles in a $3$-edge-colored $K_n$ is a problem of Erd\H{o}s and S\'{o}s (see~\cite{MR0337636}) that was solved by flag algebras~\cite{MR3667664}. The asymptotic construction is an iterated blow-up of a properly $3$-edge-colored $K_4$. 
Properly $3$-edge-colored $K_4$ can be embedded as a tetrahedron in $\mathbb{R}^3$. 
This gives $\frac{1}{16}n^3(1+o(1))$ $\varepsilon$-isomorphic triangles in $\mathbb{R}^3$. 
This heuristics suggests that increasing the dimension beyond $3$ may allow us to embed slightly more $\varepsilon$-isomorphic triangles by making it possible to embed more of the iterated blow-up of $K_4$ construction.  
The number of rainbow triangles 
the iterated blow-up of a properly $3$-edge-colored $K_4$ is $\frac{1}{15}n^3(1+o(1))$ which is an upper bound on the number of $\varepsilon$-isomorphic triangles for any $d$.

In our construction maximizing the number of $\varepsilon$-similar triangles for $d=3$, the majority of triangles are actually $\varepsilon$-isomorphic. Already for $d=3$, we can embed $\frac{1}{15}n^3(1+o(1))$  $\varepsilon$-similar triangles, which is the upper bound on the number of $\varepsilon$-isomorphic triangles for any $d$. This suggests that increasing the dimension beyond $d=3$ may result in only very small increases on the number $\varepsilon$-isomorphic triangles or a very different construction is needed.

The above heuristic does not apply to isosceles triangles. Maximizing the number of $\varepsilon$-isomorphic triangles would correspond to a $2$-edge-coloring of $K_n$ and maximizing the number of induced path on $3$ vertices in one of the two colors.
The extremal construction is a balanced complete bipartite graph in one color. Increasing the dimension helps with embedding a bigger $2$-edge-coloring of $K_n$ and in turn obtaining larger number of $\varepsilon$-isomorphic triangles
with $\frac{1}{8}n^3(1+o(1))$ being the upper bound. 

In general, the number obtuse triangles do not seem to benefit as much from higher dimensions.  Embedding three $\varepsilon$-similar obtuse triangles on $4$ points is not possible for any $d$
for almost all obtuse triangles.
This contrasts with acute triangles, where $4$ points can give four $\varepsilon$-isomorphic triangles for dimension at least $3$. 
The reader may try it 
for $\varepsilon$-isomorphic triangles with cutouts in Figure~\ref{fig:cutout}.
We have not explored the above problems for obtuse triangles further.


\bibliographystyle{abbrv}
\bibliography{pointsinplane}

\end{document}